\documentclass[final,leqno,onefignum,onetabnum]{siamltex}
\usepackage{amsmath}
\usepackage{amsfonts}
\usepackage{mathrsfs}
\usepackage{amssymb}
\usepackage{verbatim}
\usepackage{url}
\usepackage{enumerate}
\usepackage[footnotesize,FIGBOTCAP,TABTOPCAP]{subfigure}
\usepackage[footnotesize]{caption}   
\usepackage{graphicx,color}
\usepackage{alltt}
\usepackage[margin=1.4in]{geometry}

\newcommand{\p}{M}

\newcommand{\B}{\mathcal{B}}

\usepackage{calc}  
\usepackage{amsfonts}
\usepackage[colorlinks=true, bookmarksopen,
            pdfcreator={pdftex},
            pdfsubject={algorithms},
            linkcolor={blue},
            anchorcolor={black},
            citecolor={red},
            filecolor={magenta},
            menucolor={black},
            pagecolor={red},
            plainpages=false,pdfpagelabels,
            urlcolor={blue}]{hyperref}
\let\oldeq\equation{}\def\equation{\par\vspace{-\parskip}\oldeq}
\newcounter{saveeqn}%

\newenvironment{example}[1][Example 1.]{\begin{trivlist}
\item[\hskip \labelsep {\bfseries #1}]}{\end{trivlist}}
\newenvironment{example2}[1][Example 2.]{\begin{trivlist}
\item[\hskip \labelsep {\bfseries #1}]}{\end{trivlist}}
\newenvironment{example3}[1][Example 3.]{\begin{trivlist}
\item[\hskip \labelsep {\bfseries #1}]}{\end{trivlist}}

\author{
 Vandana Sharma\footnotemark[1] \and Jeff Morgan\footnotemark[2]}

\begin{document}
\title{Uniform bounds for solutions to Volume-Surface reaction diffusion systems}
\maketitle

\renewcommand{\thefootnote}{\fnsymbol{footnote}}

 \footnotetext[1]{Department of Mathematical and Statistical Sciences, Arizona State University, Tempe, USA, AZ 85281. Email: \textup{\nocorr \texttt{vandanas@asu.edu}}.}
\footnotetext[2] {Department of Mathematics, The University of Houston, Houston, USA,  TX 77004. Email: \textup{\nocorr \texttt{jjmorgan@central.uh.edu}}. 
 }

\begin{abstract}
We consider a reaction-diffusion system where some components
react and diffuse on the boundary of a region, while other components diffuse in
the interior and react with those on the boundary through mass transport. We establish criteria guaranteeing that solutions are uniformly bounded in time.
\end{abstract}

\begin{keywords}
 reaction-diffusion equations, mass transport, conservation of mass,  Laplace Beltrami operator, uniform estimates, a priori estimates.
\end{keywords}

\begin{AMS}
35K57, 35B45
\end{AMS}

\section{Introduction}
Reaction-diffusion systems have been an intensive area of research in pure and applied mathematics, chemisty, biology, and physics. One of the challenging problems is to obtain uniform bounds for the solutions.  

One of the approach used  by researchers in the past  is to bootstrap a priori $L_1$ estimates to $L_{\infty}$ estimates. Considering the smoothing properties of parabolic systems, it seems reasonable to expect $L_{\infty}$ estimates from $L_1$ estimates. But Pierre and Schmidt \cite{RefWorks:124} have shown that solutions can posess $L_1$ estimates and can still blow up in finite time in the $L_{\infty}$ norm.  

This work is concerned with uniform bounds for solutions to reaction-diffusion systems of the form 
\begin{align}\label{sy5}
 u_t\nonumber&=  D\Delta u+H(u)
 & x\in \Omega, \quad&0<t<T
\\\nonumber v_t&=\tilde D\Delta_{M} v+F(u,v)& x\in M,\quad& 0<t<T\\ D\frac{\partial u}{\partial \eta}&=G(u,v) & x\in M, \quad&0<t<T\\\nonumber
u&=u_0  &x\in\Omega ,\quad& t=0\\\nonumber v&=v_0 & x\in M ,\quad &t=0\end{align}
where 
$\Omega$ is a bounded domain in $\mathbb{R}^n$ with smooth boundary M ($\partial \Omega$) belonging to the class $C^{2+\mu}$ with $\mu>0$ such that $\Omega$ lies locally on one side of its  boundary. $\eta$ is the unit outward normal to $M$ (from $\Omega$), and $\Delta$ and $\Delta_{\p}$ are the Laplace and the Laplace Beltrami operators, respectively. In addition, $m, k, n,i$ and $ j$ are positive integers, $D$ and $\tilde D$ are $k\times k$ and $m\times m$ diagonal matrices with positive diagonal entries $\lbrace d_j\rbrace_{1\leq j\leq k}$ and $\lbrace\tilde d_i\rbrace_{1\leq i\leq m}$, respectively. $F=(F_i):\mathbb{R}^k\times\mathbb{R}^m\rightarrow \mathbb{R}^m, G=(G_j):\mathbb{R}^k\times\mathbb{R}^m\rightarrow \mathbb{R}^k$ and $H=(H_j):\mathbb{R}^k \rightarrow \mathbb{R}^k$, and
$ u_0=( {u_0}_j)\in W_p^{2}(\Omega)$ and $v_0= ({v_0}_i)\in W_p^{2}(M)$ with $p>n$. Also, $u_0 $  and $v_0$ satisfy the compatibility condition\[ D{\frac{ \partial {u_0}}{\partial \eta}} =G(u_0,v_0)\quad \text{on $M$}\] Systems of this type have been studied in \cite{RefWorks:005}, \cite{RefWorks:003} and \cite{RefWorks:99}.

Additional conditions are placed on $F$, $G$ and $H$ in section 2, and we see in section 4 that these conditions occur naturally in applications. These hypotheses were recently used to obtain global existence in \cite{RefWorks:2015}.

In general, system ($\ref{sy5}$) is somewhat reminiscent of two component systems where both of the unknowns react and diffuse inside $\Omega$, with various homogeneous boundary conditions and nonnegative initial data. In that setting, global well-posedness and uniform boundedness has been studied by many researchers (cf \cite{RefWorks:11}, \cite{RefWorks:123}), and we refer the interested reader to the excellent survey of Pierre \cite{ RefWorks:86}.

Our work is organized as follows. Section 2 contains hypotheses and statements of the main results, section 3 contains foundational definitions and proofs, and section 4 contains some examples.
\section{Statements of Main Results}
\quad\\

\begin{definition}\label{blah}
A function $(u,v)$ is said to be a $\it solution$ of $\left (\ref{sy5}\right)$ if and only if \[u \in C(\overline \Omega\times[0,T),\mathbb{R}^k)\cap C^{1,0}(\overline \Omega\times(0,T),\mathbb{R}^k)\cap C^{2,1}( \Omega\times(0,T),\mathbb{R}^k)\] and \[v \in C(M\times[0,T),\mathbb{R}^m)\cap C^{2,1}( M\times(0,T),\mathbb{R}^m) \] such that $(u,v)$ satisfies $\left (\ref{sy5}\right)$. If $T=\infty$, the solution is said to be a {\it global solution.}
\end{definition}
Moreover, a solution $(u,v)$ defined for $0\leq t<b$ is a $\it  maximal\  solution$ of $\left (\ref{sy5}\right)$ if and only if $(u,v)$ solves $\left (\ref{sy5}\right)$ with $T=b$, and if $d>b$ and $(\tilde u,\tilde v)$ solves $\left (\ref{sy5}\right)$ for $T=d$, then there exists $0<c<b$ such that $(u(\cdot,c),v(\cdot,c))\ne(\tilde u(\cdot,c), \tilde v(\cdot,c))$.\\

The positive orthant of $\mathbb{R}^n$ is defined by $\lbrace z\in \mathbb{R}^n : z_i\geq 0\rbrace$, and denoted by ${\mathbb{R}}^{n}_{+}$. We say $F$, $G$ and $H$ are $\it quasi positive$ if and only if $F_i(\zeta,\xi)\geq 0$ whenever $\xi\in\mathbb{R}_+^m$ and $\zeta\in \mathbb{R}_+^k$ with $\xi_i=0$ for $i=1,...,m$, and $G_j( \zeta,\xi)$, $H_j( \zeta)\geq 0$ whenever $\xi\in\mathbb{R}_+^m$ and $ \zeta\in\mathbb{R}_+^k$ with $\zeta_j=0$, for $j=1,...,k.$\\

The purpose of this study is to give sufficient conditions guaranteeing that $\left (\ref{sy5}\right)$ has a uniformly bounded global solution. We gave conditions in \cite{RefWorks:2015} that guarantee global solutions to $\left (\ref{sy5}\right)$, and we use these and additional hypotheses to obtain the results in this work. \\

For $l\in \lbrace 1, 2, 3\rbrace$, we say condition $V_{i,j}l$ holds for $1\leq j\leq k$ and $1\leq i\leq m$ if and only if \\

\begin{itemize}
\item[($V_{i,j}1$)] There exist $\sigma>0,\beta>0$, and $\alpha>0$ such that \[\sigma F_i(\zeta,\nu)+ G_j(\zeta,\nu)\leq \alpha(\zeta_j+\nu_i+1)\quad\text{and}\quad H_j(\zeta)\leq \beta(\zeta_j+1)\quad\text{ for all} \quad\nu \in\mathbb{R}^m_{+},\  \zeta \in\mathbb{R}^k_{+}\]
\item[($V_{i,j}2$)]There exists $K_g>0$ such that \[\quad G_j(\zeta,\nu)\leq K_g(\zeta_j+\nu_i+1)\quad\text{for all}\quad\nu \in\mathbb{R}^m_{+},\  \zeta \in\mathbb{R}^k_{+}\]
\item[($V_{i,j}3$)]There exists $l \in \mathbb{N}$ and $K_f>0$ such that \[\quad  F_i(\zeta,\nu)\leq K_f( |\zeta|+|\nu|+1)^l\quad\text{for all} \quad\nu \in\mathbb{R}^m_{+},\  \zeta \in\mathbb{R}^k_{+}\]
\end{itemize}
\vspace{.3cm}
In addition, we say $V_{i,j}$ holds if $(V_{i,j}1), (V_{i,j}2)$ and $(V_{i,j}3)$ hold. We showed in \cite{RefWorks:2015} that $(V_{i,j}1)$ provides $L_1$ estimates for $u_j$ on $\Omega$ and $M \times(\tau,T)$, and $v_i$ on $M$. $(V_{i,j}2)$ helps us bootstrap to better $L_p$ estimates on $M\times(\tau,T)$ and $\Omega\times(\tau,T)$. Finally, both $(V_{i,j}2)$ and $(V_{i,j}3)$ allow us to obtain sup norm estimates on $u_j$ and $v_i$, once $L_p$ estimates have been obtained for all components of $u$ and $v$. \\

The following result is given in  \cite{RefWorks:2015}.\\

\begin{theorem}\label{great}
Suppose $F$, $G$ and $H$ are locally Lipschitz,  quasi positive, and $u_0, v_0$ are componentwise nonnegative functions. Also, assume that for each $1\leq j\leq k$ and $1\leq i\leq m$, there exists $l_i\in\lbrace 1,...,k\rbrace$ and $k_j\in\lbrace1,...,m\rbrace$ so that both $V_{i,l_i}$ and $V_{k_j,j}$ are satisfied. Then  $(\ref{sy5})$ has a unique component-wise nonegative global solution.
\end{theorem}\\

Our main results are given below.\\

\begin{theorem}\label{globalu}
Suppose $F$, $G$ and $H$ are locally Lipschitz,  quasi positive, and $u_0, v_0$ are componentwise nonnegative functions. Also, assume that for each $1\leq j\leq k$ and $1\leq i\leq m$, there exists $l_i\in\lbrace 1,...,k\rbrace$ and $k_j\in\lbrace1,...,m\rbrace$ so that both $V_{i,l_i}$ and $V_{k_j,j}$ are satisfied. If there exists $L>0$ independent of $\tau\geq0$ such that \[\int_{\tau}^{\tau+1}\int_{\Omega} u_j+\int_{\tau}^{\tau+1}\int_{M} v_i+\int_{\tau}^{\tau+1}\int_{M} u_j \leq L\]  for all $\tau\geq 0$, then the solution of $(\ref{sy5})$ is uniformly bounded in the sup-norm.
\end{theorem}\\

\begin{corollary}\label{adventure23}
Assume the hypothesis of  Theorem $\ref{great}$ hold with $\alpha=0$ and $\beta=0$. Then there exists $K>0$ independent of $t$ such that \[ \Vert u(\cdot,t)\Vert_{\infty,\Omega}+\Vert v(\cdot,t)\Vert_{\infty,M}\leq K \quad\text {for all } \ t\geq 0\]

\end{corollary}

The condition in $(V_{i,j}1)$ above is related to the idea of conservation of mass in \cite{RefWorks:86} in the case when $\alpha=\beta=0$. Although all of the systems in that work reacted and diffused in $\Omega$, it is still interesting to note that the same condition leads to the uniform $L_1$ estimates required above, and consequently uniform sup norm estimates for solutions to  $(\ref{sy5})$.\\


\section{ Definitions, Estimates and Proofs of Main Results}
\setcounter{equation}{0}
\subsection{Definitions of Basic Function Spaces}
Throughout this work, $n\geq 2$. Let $\Omega$  be a bounded domain on $\mathbb{R}^n$ with smooth boundary such that $\Omega$ lies locally on one side of $\partial\Omega$. We define all function spaces on $\Omega$ and $\Omega_T=\Omega\times(0,T)$. 
$L_p(\Omega)$ is the Banach space consisting of all measurable functions on $\Omega$ that are $p^{th}(p\geq 1)$ power summable on $\Omega$. The norm is defined as\[ \Vert u\Vert_{p,\Omega}=\left(\int_{\Omega}| u(x)|^p dx\right)^{\frac{1}{p}}\]
Also, \[\Vert u\Vert_{\infty,\Omega}= ess \sup\lbrace |u(x)|:x\in\Omega\rbrace.\]
Measurability and summability are to be understood everywhere in the sense of Lebesgue.

If $p\geq 1$, then $W^2_p(\Omega)$ is the Sobolev space of functions $u:\Omega\rightarrow \mathbb{R}$ with generalized derivatives, $\partial_x^s u$ (in the sense of distributions) $|s|\leq 2$ belonging to $L_p(\Omega)$.  Here $s=(s_1,s_2,$...,$s_n),|s|=s_1+s_2+..+s_n$, $|s|\leq2$, and $\partial_x^{s}=\partial_1^{s_1}\partial_2^{s_2}$...$\partial_n^{s_n}$ where $\partial_i=\frac{\partial}{\partial x_i}$. The norm in this space is \[\Vert u\Vert_{p,\Omega}^{(2)}=\sum_{|s|=0}^{2}\Vert \partial_x^s u\Vert_{p,\Omega} \]

Similarly, $W^{2,1}_p(\Omega_T)$ is the Sobolev space of functions $u:\Omega_T\rightarrow \mathbb{R}$ with generalized derivatives, $\partial_x^s\partial_t^r u$ (in the sense of distributions) where $2r+|s|\leq 2$ and each derivative belonging to $L_p(\Omega_T)$. The norm in this space is \[\Vert u\Vert_{p,\Omega_T}^{(2)}=\sum_{2r+|s|=0}^{2}\Vert \partial_x^s\partial_t^r u\Vert_{p,\Omega_T} \]

We also introduce $W^l_p(\Omega)$, where $l>0$ is not an integer, because initial data will be taken from these spaces. The space $W^l_p(\Omega)$ with nonintegral $l$, is a Banach space consisting of elements of $W^{[l]}_p$ ([$l$] is the largest integer less than $ l$) with the finite norm\[\Vert u\Vert_{p,\Omega}^{(l)} =\langle u\rangle_{p,\Omega}^{(l)}+\Vert u\Vert_{p,\Omega}^{([l])} \]
where  \[\Vert u\Vert_{p,\Omega}^{([l])}=\sum_{s=0}^{[l]}\Vert \partial_x^s u\Vert_{p,\Omega} \] and 
\[\langle u\rangle_{p,\Omega}^{(l)}=\sum_{s=[l]}\left(\int_\Omega dx\int_\Omega{|\partial_x^s u(x)-\partial_y^s u(y)|}^p.\frac{dy}{|x-y|^{n+p(l-[l])}}\right)^\frac{1}{p}\]
$W^{l,\frac{l}{2}}_p(\partial\Omega_T)$ spaces with non integral $l$ also play an important role in the study of boundary value problems with nonhomogeneous boundary conditions, especially in the proof of exact estimates of their solutions.  It is a Banach space when $p\geq 1$, which is defined by means of parametrization of the surface $\partial\Omega$. For a rigorous treatment of these spaces, we refer the reader to page 81 of Chapter 2 of \cite{RefWorks:65}.

\subsection{Bootstrapping Strategy}

For completeness of  our arguments, we state the results below which will help us obtain $L_q$ estimates for the solution to $(\ref{sy5})$. The following system will play a central role in duality arguments.\begin{align}\label{aj2}
\Psi_t & =-\tilde d\Delta_M \Psi-\tilde\vartheta & (x,t)\in M\times (\tau,T)\nonumber\\ 
\Psi &= 0 & x\in M , t=T\tag{3a}\nonumber
\end{align}
\begin{align}\label{ajj3}
\varphi_t &=- d\Delta \varphi-\vartheta & (x,t)\in \Omega\times (\tau,T)\nonumber\\ \kappa_1d\frac{\partial \varphi}{\partial \eta}&+\kappa_2\varphi=\Psi & (x,t)\in M\times (\tau,T)\tag{3b}\nonumber\\
\varphi&=0 &x\in\Omega ,\quad t=T\nonumber\end{align}
Here, $q>1$, $0<\tau<T$, $\tilde\vartheta\in L_{q}{(M\times(\tau,T))}$ and $\tilde\vartheta\geq 0$, and  $\vartheta\in L_{q}{(\Omega\times(\tau,T))}$ and $\vartheta\geq 0$. Also $d>0$, $\tilde d>0$, and $\kappa_1,\kappa_2 \in\mathbb{R}$ such that $\kappa_1\geq 0$ and $|\kappa_1|+|\kappa_2|> 0$. Lemmas $\ref{manifold}$ to $\ref{i}$ provide helpful estimates. See section 4 and 5 of \cite{RefWorks:2015} for the proof of Lemmas $\ref{manifold}$, $\ref{flat}$, and $\ref{bw}$. Lemmas $\ref{con}$ and $\ref{Hol}$ follow from Lemma 3.3 of chapter 2 in \cite{RefWorks:65}, and Lemma $\ref{i}$ can be found on page 49 in \cite{RefWorks:69}.   \\

\begin{lemma}\label{manifold}
$(\ref{aj2})$ has a unique nonnegative solution $\Psi\in{W_{q}}^{2,1}{(M\times(\tau,T))}$ and there exists $C_{q,T-\tau}>0$ independent of $\tilde\vartheta$ such that \[\Vert \Psi\Vert_{q,M\times(\tau,T)}^{(2)}\leq C_{q,T-\tau}\Vert\tilde\theta\Vert_{q,M\times(\tau,T)}\]
Furthermore, if $\kappa_1=0$ and $\kappa_2>0$, then $(\ref{ajj3})$ has a unique nonnegative solution $\varphi\in W_{q}^{2,1}{(\Omega\times(\tau,T))}$. Moreover, there exists $C_{q,T-\tau}>0$ independent of $\vartheta$ and $\tilde\vartheta$, and dependent on $d,\tilde d,\kappa_1$ and $\kappa_2$ such that
\[{\Vert \varphi\Vert}_{q,\Omega\times(\tau,T)}^{(2)}\leq C_{q,T-\tau}(\Vert\tilde\theta\Vert_{q,M\times(\tau,T)}+\Vert\theta\Vert_{q,\Omega\times(\tau,T)})\]
\end{lemma}\\
\begin{lemma}\label{con}
Suppose $1<q\leq r\leq \frac{(n+2)q}{(n+2)-2q}$. Then there exists $\tilde c>0$ depending on $ n, q, T-\tau$ and $\Omega$ such that if $\varphi\in W_q^{2,1}(\Omega\times(\tau,T))$, then 
\[  \Vert\varphi\Vert _{r,\Omega\times(\tau,T)}\leq \tilde c \Vert \varphi\Vert^{(2)}_{q,\Omega\times(\tau,T)}\] 
\end{lemma}\\
\begin{lemma}\label{flat}
Let $1<q< n+2$ and $1<r\leq \frac{(n+1)q}{n+2-q}$. There exists a constant $\hat C>0$ depending on $q,T-\tau,M$ and $n$   such that if $\varphi\in W^{2,1}_q(\Omega\times(\tau,T))$, then \[\left\Vert\frac{\partial\varphi}{\partial\eta}\right\Vert_{r,M\times(\tau,T)}\leq \hat C {\left \Vert \varphi\right\Vert}^{(2)}_{q,\Omega\times(\tau,T)}\]

\end{lemma}

\begin{lemma}\label{Hol}Suppose $q>\frac{n+2}{2}$. There exists a constant $\hat c_1, \hat c_2$ depending on $n, q$,$T- \tau$, and $\Omega$ such that if $\varphi \in W_q^{2,1}({\Omega\times(\tau,T)})$ and  $\Psi \in W_q^{2,1}({M\times(\tau,T)})$, then \[ \Vert\varphi \Vert_{\infty,\Omega\times(\tau,T)}\leq \hat c_1 \Vert \varphi\Vert^{(2)}_{q,\Omega\times(\tau,T)}\quad \text{and}\quad \Vert\Psi \Vert_{\infty,M\times(\tau,T)}\leq \hat c_2 \Vert \Psi\Vert^{(2)}_{q,M\times(\tau,T)} \] Moreover, if $q>n+2$, there exists a constant $\hat c_3>0$ depending on $q,T-\tau,M$ and $n$ such that if $\varphi \in W_q^{2,1}({\Omega\times(\tau,T)})$, then \[\left\Vert\frac{\partial\varphi}{\partial\eta}\right\Vert_{\infty,M\times(\tau,T)}\leq \hat c_3 {\left \Vert \varphi\right\Vert}^{(2)}_{q,\Omega\times(\tau,T)}\] 
\end{lemma}\\
\begin{lemma}\label{bw}
Suppose $q>n+1$ and $T>0$ and $\theta\in L_q(\Omega\times(0, T))$, $\gamma\in L_q(M\times(0, T))$ and  $\varphi_0\in W^{2}_q(\Omega)$ such that  \[  d\frac{\partial {\varphi_0}}{\partial \eta} =\gamma(x,0)\quad{\text {on $M$}}\] Then the unique solution to \begin{align}\label{bro}
 \varphi_t\nonumber&= d\Delta \varphi +\theta&
 x\in \Omega,\quad &0<t< T
\\ d\frac{\partial \varphi}{\partial \eta}&=\gamma &x\in M,\quad &0<t< T\\\nonumber
\varphi&= \varphi_0 & x\in\Omega ,\quad& t=0
\end{align}is continuous on  ${\overline\Omega}\times[0,  T]$, and there exists $C_{p,T}>0$ independent of $\theta, \gamma$ and $\varphi_0$ such that \[ \Vert \varphi\Vert_{\infty,\Omega_{ T}}\leq C_{p,T}( \Vert \theta\Vert_{q,\Omega_{ T}}+\Vert \gamma\Vert_{q,M_{ T}}+\Vert \varphi_0\Vert^{(2)}_{q,\Omega})\]
\end{lemma}

\begin{lemma}\label{i}
If $1<q< n$, then $ W^1_q(\Omega)$ embedds continuously into $W_q^{(1-\frac{1}{q} )}(\partial\Omega)$ for $q\leq p\leq q^*=\frac{nq}{n-q}$. In addition, if $\epsilon>0$, there exists $C_{\epsilon}>0$ such that  

\[\Vert v\Vert^2_{2,\partial\Omega}\leq \epsilon {\Vert  v_{x}\Vert}^2_{2,\B }+C_{\epsilon}{\Vert  v\Vert}^2_{2,\Omega }\] for all $v\in W^1_2(\Omega)$.
\end{lemma}\\

The following lemma allows us to bootstrap $L_1$ estimates to better $L_p$ estimates.
\vspace{.2cm}
\begin{lemma}\label{global_p}
Suppose $\tau\geq 0$ and $q\geq 1$, and $1\leq j\leq k$ and $1\leq i\leq m$, so that both $V_{i,j}1$ and $V_{i,j}2$ are satisfied. If $u_j\in L_q(\Omega\times(\tau,\tau+3))$, $u_j,v_i\in L_q(M\times(\tau,\tau+3))$, and $p=\left(\frac{n+3}{n+2}\right)q$, then there exists $C_{p}>0$ independent of $\tau$ such that  
\begin{align*}
\Vert u_j\Vert_{p,\Omega\times(\tau+1,\tau+3)}&+\Vert u_j\Vert_{p,M\times(\tau+1,\tau+3)}+\Vert v_i\Vert_{p,M\times(\tau+1,\tau+3)}\\
&\leq C_p \left(\Vert u_j\Vert_{q,M\times(\tau,\tau+3)}+\Vert u_j\Vert_{q,\Omega\times(\tau,\tau+3)}+\Vert v_i\Vert_{q,M\times(\tau,\tau+3)}+1\right)
\end{align*}
\end{lemma}
\begin{proof} 
Note that $p>1$. Set $p'=\frac{p}{p-1}$, and consider $(\ref{aj2})$, $(\ref{ajj3})$ on $(\tau,\tau+3)$, with $\kappa_1=0$, $\kappa_2=1$, $\tilde\vartheta, \vartheta\geq 0$, $ {\Vert \tilde\vartheta\parallel}_{p',(M\times(\tau,\tau+3))}=1$ and $ {\Vert \vartheta\parallel}_{p',(\Omega\times(\tau,\tau+3))}=1$. Let $\varphi$ and $\Psi$ be the nonnegative solutions to $(\ref{aj2})$ and $(\ref{ajj3})$, and define a cut off function $\psi\in{C_{0}}^{\infty}(\mathbb{R},[0,1])$  such that $\psi(t)=1$ for all $t\geq \tau+\frac{1}{2}$ and $\psi(t)=0$ for all $t\leq \tau$. In addition, define $
w(x,t)= \psi(t)\Psi(x,t) $ and
$ z(x,t)= \psi(t)\varphi(x,t)$.
 From construction, $ w(x,t)= \Psi(x,t)$ and $z(x,t)= \varphi(x,t)$
 for all $(x,t)\in M\times (\tau+\frac{1}{2}, \tau+3)$ and $(x,t)\in \Omega\times (\tau+\frac{1}{2}, \tau+3)$ respectively. Also $w$ and $z$ satisfy the system 
\begin{align}\label{aj33} z_t&=- d\Delta z-\psi(t)\vartheta+\psi'(t)\varphi(x,t)
&( x,t)\in \Omega\times (\tau,\tau+3)\nonumber
\\ w_t&=-\tilde d\Delta_M w-\psi(t)\tilde\vartheta+\psi'(t)\Psi(t) & (x,t)\in M\times (\tau,\tau+3)\nonumber\\ w&=z &(x,t)\in M\times (\tau,\tau+3)\\
z&=0 &x\in\Omega ,\quad t=\tau+3\nonumber\\ w&=0 &x\in M ,\quad t=\tau+3\nonumber\end{align}
Then
\begin{align*}
\int_\tau^{\tau+3}\int_{\Omega} u_j\psi\vartheta &+\int_\tau^{\tau+3}\int_{M} v_i\psi\tilde\vartheta  \\ =\int_\tau^{\tau+3}\int_{\Omega} u_j(&-z_t-d\Delta z+\psi'\varphi) +\int_\tau^{\tau+3}\int_{M} v_i(-w_t-\tilde d\Delta_M w+\psi' \Psi) \\
=\int_\tau^{\tau+3}\int_{\Omega}(z (&{u_j}_t-d\Delta u_j)+u_j\psi'\varphi) +\int_\tau^{\tau+3}\int_{M} (w({v_i}_t-\tilde d\Delta_M v_i)+v_i\psi'\Psi )\\ -d\int_\tau^{\tau+3}\int_{M} &u_j\frac{\partial z}{\partial\eta}+d\int_\tau^{\tau+3}\int_{M} \frac{\partial u_j}{\partial\eta}z
\end{align*}
since $w(\cdot,\tau+3)=0$, $z(\cdot,\tau+3)=0$, $w(\cdot,\tau)=0$ and $z(\cdot,\tau)=0$. So, from $(\ref{sy5})$,
\begin{align*}
\int_{\tau+\frac{1}{2}}^{\tau+3}\int_{\Omega} u_j\vartheta +\int_{\tau+\frac{1}{2}}^{\tau+3}\int_{M} v_i\tilde\vartheta  &\leq \int_\tau^{\tau+3}\int_M  (F_j(u,v)+G_i(u,v))w+\int_\tau^{\tau+3}\int_{\Omega} z H_j(u)\\& \quad\quad+\int_\tau^{\tau+3}\int_{M}v_i\psi'\Psi + \int_\tau^{\tau+3}\int_{\Omega}  u_j\psi'\varphi \\& \quad\quad-d\int_\tau^{\tau+3}\int_{M} u_j\frac{\partial z}{\partial\eta}
\end{align*}
Using $(V_{i,j}1)$,
\begin{align}\label{ineqn}
\int_{\tau+\frac{1}{2}}^{\tau+3}\int_{\Omega} u_j\vartheta +\int_{\tau+\frac{1}{2}}^{\tau+3}\int_{M} v_i\tilde\vartheta \nonumber &\leq \int_\tau^{\tau+3}\int_M \alpha(u_j+v_i+1)w+\int_\tau^{\tau+3}\int_{\Omega}\beta(u_j+1)z \\ & \quad\quad + \int_\tau^{\tau+3}\int_{\Omega}  u_j\psi'\varphi+ \int_\tau^{\tau+3}\int_{M}v_i\psi'\Psi -d\int_\tau^{\tau+3}\int_{M} u_j\frac{\partial z}{\partial\eta}
\end{align}
Now we break the argument in two cases.

Case 1: Suppose $q=1$. Then $u_j\in L_1(\Omega\times(\tau,\tau+3))$, $u_j,v_i\in L_1(M\times(\tau,\tau+3))$, $p=\frac{n+3}{n+2}$, and $ p'=n+3$. From Lemma $\ref{manifold}$ and $\ref{Hol}$ and our hypothesis on $\vartheta$ and $\tilde\vartheta$, $\Psi$, $\varphi$ and $\frac{\partial z}{\partial\eta}$ are sup norm bounded independent of $\tau$. Application of H\"older's inequality in $(\ref{ineqn})$ implies there exists $\tilde C_{p}>0$, independent of $\tau$, such that
\[\int_{\tau+\frac{1}{2}}^{\tau+3}\int_{\Omega} u_j\vartheta+\int_{\tau+\frac{1}{2}}^{\tau+3}\int_M v_i\tilde\vartheta\leq \tilde C_p (\Vert u_j\Vert_{1,\Omega\times(\tau,\tau+3)}+\Vert v_i\Vert_{1,M\times(\tau,\tau+3)}+\Vert u_j\Vert_{1,M\times(\tau,\tau+3)}+1)\]
Therefore, from duality, 
\begin{align*}
\Vert u_j\Vert_{p,\Omega\times (\tau+\frac{1}{2},\tau+3)}&+\Vert v_i\Vert_{p,M\times(\tau+\frac{1}{2},\tau+3)}\\&\leq \tilde C_{p}(\Vert u_j\Vert_{1,\Omega\times(\tau,\tau+3)}+\Vert v_i\Vert_{1,M\times(\tau,\tau+3)}+\Vert u_j\Vert_{1,M\times(\tau,\tau+3)}+1)
\end{align*}

Case 2: Suppose $q>1$. Then $u_j\in L_q(\Omega\times(\tau,\tau+3))$ and $u_j,v_i\in L_q(M\times(\tau,\tau+3))$. Applying H\"older's inequality in $(\ref{ineqn})$, and using Lemmas $\ref{manifold}$, $\ref{con}$ and $\ref{flat}$ and our hypothesis on $\vartheta$ and $\tilde\vartheta$, implies there exists $\tilde C_{p}>0$, independent of $\tau$, such that 
\[\int_{\tau+\frac{1}{2}}^{\tau+3}\int_{\Omega} u_j\vartheta+\int_{\tau+\frac{1}{2}}^{\tau+3}\int_M v_i\tilde\vartheta\leq \tilde C_p (\Vert u_j\Vert_{q,\Omega\times(\tau,\tau+3)}+\Vert v_i\Vert_{q,M\times(\tau,\tau+3)}+\Vert u_j\Vert_{q,M\times(\tau,\tau+3)}+1)\]
Therefore, from duality, 
\begin{align*}
\Vert u_j\Vert_{p,\Omega\times (\tau+\frac{1}{2},\tau+3)}&+\Vert v_i\Vert_{p,M\times (\tau+\frac{1}{2},\tau+3)}\\&\leq \tilde C_{p}(\Vert u_j\Vert_{q,\Omega\times(\tau,\tau+3)}+\Vert v_i\Vert_{q,M\times(\tau,\tau+3)}+\Vert u_j\Vert_{q,M\times(\tau,\tau+3)}+1)
\end{align*}

Now, it remains to show that $u_j\in L_{p}(M\times (\tau+1,\tau+3))$. Define a cut off function $\tilde\psi\in{C_{0}}^{\infty}(\mathbb{R},[0,1])$  such that $\tilde\psi(t)=1$ for all $t\geq \tau+1$ and $\tilde\psi(t)=0$ for all $t\leq \tau+\frac{1}{2}$. Consider $U_j=\tilde\psi u_j$, and by definition of $\tilde\psi$, $U_j=u_j$ for all $t\geq \tau+1$ and $U_j(\cdot, \tau+\frac{1}{2})=0$. Multiplying the ${U_j}_t$ equation by $U_j^{p-1}$, and integrating over $\Omega\times (\tau+\frac{1}{2},\tau+3)$, we get 
\begin{align}
\int_{\tau+\frac{1}{2}}^{\tau+3}\int_{\Omega}U_j^{p-1}{U_j}_t &=d_j\int_{\tau+\frac{1}{2}}^{\tau+3}\int_{\Omega}U_j^{p-1}\Delta U_j+\int_{\tau+\frac{1}{2}}^{\tau+3}\int_{\Omega}U_j^{p-1}\tilde\psi H_j(u)-\int_{\tau+\frac{1}{2}}^{\tau+3}\int_{\Omega}U_j^{p-1}\tilde\psi'(t) u_j\nonumber\\ &= d_j\int_{\tau+\frac{1}{2}}^{\tau+3}\int_{M}U_j^{p-1}{\frac{\partial U_j}{\partial\eta}}-d_j\int_{\tau+\frac{1}{2}}^{\tau+3}\int_{\Omega}(p-1)U_j^{p-2}{|\nabla U_j|^{2}}+ \int_{\tau+\frac{1}{2}}^{\tau+3}\int_{\Omega}U_j^{p-1}\psi H_j(u)\nonumber\\
&\quad  -\int_{\tau+\frac{1}{2}}^{\tau+3}\int_{\Omega}U_j^{p-1}\tilde\psi'(t) u_j \nonumber
\end{align}
Using $(V_{i,j}2)$ and $\vert\tilde\psi'(t)\vert\leq K$, we get
\begin{align}\label{MM3}
 \int_{\Omega}\frac{U_j^{p}(\cdot,\tau+3)}{p}&+d_j\int_{\tau+\frac{1}{2}}^{\tau+3}\int_{\Omega}\frac{4(p-1)}{p^{2}}{|\nabla U_j^{\frac{p}{2}}|^{2}}\nonumber\\&\leq K_g\int_{\tau+\frac{1}{2}}^{\tau+3}\int_{M}U_j^{p-1}(u_j+v_i+1)+\beta\int_{\tau+\frac{1}{2}}^{\tau+3}\int_{\Omega} (u_j+1)U_j^{p-1}+K \int_{\tau+\frac{1}{2}}^{\tau+3}\int_{\Omega} u_jU_j^{p-1}\nonumber\\
&\leq K_g\left(\int_{\tau+\frac{1}{2}}^{\tau+3}\int_{M}u_jU_j^{p-1}+v_i U_j^{p-1}+U_j^{p-1}\right)+(\beta+K )\int_{\tau+\frac{1}{2}}^{\tau+3}\int_{\Omega} u_jU_j^{p-1}\nonumber\\&\quad+\beta\int_{\tau+\frac{1}{2}}^{\tau+3}\int_{\Omega} U_j^{p-1}
\end{align}
Applying Young's inequality in $(\ref{MM3})$ gives
\begin{align}\label{MM84}
 \int_{\Omega}\frac{U_j^{p}(\cdot, \tau+3)}{p}+d_j\int_{\tau+\frac{1}{2}}^{\tau+3}\int_{\Omega}\frac{4(p-1)}{p^{2}}{|\nabla U_j^{\frac{p}{2}}|^{2}}&\leq K_g\left(\frac{3p-2}{p}\right)\int_{\tau+\frac{1}{2}}^{\tau+3}\int_{M}U_j^{p}+\frac{(2\beta+K)(p-1)}{p} \int_{\tau+\frac{1}{2}}^{\tau+3}\int_{\Omega} U_j^p\nonumber\\ &\quad \quad +|\Omega|\frac{5\beta}{2p}+ K_g\left(\frac{1}{p}\right)\int_{\tau+\frac{1}{2}}^{\tau+3}\int_{M}v_i^{p}+\frac{5|M| K_g}{2p}\nonumber\\&\quad\quad+\frac{(\beta+K)}{p} \int_{\tau+\frac{1}{2}}^{\tau+3}\int_{\Omega} u_j^p
\end{align}
Also, from Lemma $\ref{i}$ for all $\epsilon>0$ and $\tau>0$, there exists $C_\epsilon>0$ independent of $\tau$, such that
\begin{eqnarray}\label{MM}
\int_{\tau+\frac{1}{2}}^{\tau+3}\int_{M}U_j^{p}\leq C_{\epsilon}\int_{\tau+\frac{1}{2}}^{\tau+3}\int_{\Omega}U_j^{p}+\epsilon\int_{\tau+\frac{1}{2}}^{\tau+3}\int_{\Omega}{|\nabla U_j^{\frac{p}{2}}|^{2}}\quad
\end{eqnarray}
To estimate $\Vert U_j\Vert_{p,M\times(\tau+\frac{1}{2},\tau+3)}$, we use $(\ref{MM84})$ to obtain
\begin{eqnarray}\label{MM4}
\epsilon \int_{\tau+\frac{1}{2}}^{\tau+3}\int_{\Omega}|\nabla U_j^{\frac{p}{2}}|^{2}& \leq \left(\frac{p^2}{4d_j(p-1)}\right)3K_g \epsilon \int_{\tau+\frac{1}{2}}^{\tau+3}\int_{M}U_j^{p}+\epsilon\left(\frac{p^2}{4d_j(p-1)}\right)\left(\frac{(2\beta+K)(p-1)}{p}\right)\int_{\tau+\frac{1}{2}}^{\tau+3}\int_{\Omega} U_j^p\nonumber\\ & +\left(\frac{p^2}{4d_j(p-1)}\right)\left(\epsilon \frac{(\beta+K)}{p} \int_{\tau+\frac{1}{2}}^{\tau+3}\int_{\Omega} u_j^{p}+\epsilon K_g\left(\frac{1}{p}\right)\int_{\tau+\frac{1}{2}}^{\tau+3}\int_{M}v_i^{p}+\epsilon\frac{5|M|K_g}{2p}+\epsilon|\Omega|\frac{5\beta}{2p}\right)\nonumber
\end{eqnarray}
Using $(\ref{MM})$, we have
\begin{align*}
\int_{\tau+\frac{1}{2}}^{\tau+3}\int_{M}U_j^{p}&\leq C_{\epsilon}\int_{\tau+\frac{1}{2}}^{\tau+3}\int_{\Omega}U_j^{p}+3K_g\left(\frac{p^2}{4d_j(p-1)}\right)\epsilon \int_{\tau+\frac{1}{2}}^{\tau+3}\int_{M}U_j^{p}\\ &\quad +\left(\frac{p^2}{4d_j(p-1)}\right)\left(\epsilon \frac{(\beta+K)}{p} \int_{\tau+\frac{1}{2}}^{\tau+3}\int_{\Omega} u_j^{p}+\epsilon K_g\left(\frac{1}{p}\right)\int_{\tau+\frac{1}{2}}^{\tau+3}\int_{M}v_i^{p}+\epsilon\frac{5|M|K_g}{2p}+|\Omega|\frac{5\beta}{2p}\right)\\&\quad+\epsilon\left(\frac{p^2}{4d_j(p-1)}\right)\left(\frac{(2\beta+K)(p-1)}{p}\right)\int_{\tau+\frac{1}{2}}^{\tau+3}\int_{\Omega} U_j^p\nonumber
\end{align*}
Recall, $u_j= U_j$ on $M\times(\tau+1,\tau+3)$. Now, choosing $\epsilon$ such that \[1-3K_g\left(\frac{p^2}{4d_j(p-1)}\right)\epsilon>0\] and using the estimates above for $u_j\in L_p(\Omega\times (\tau+\frac{1}{2},\tau+3))$ and $v_i\in L_p(M\times (\tau+\frac{1}{2},\tau+3))$, independent of $\tau$, we have a bound for $u_j\in L_p(M\times(\tau+1,\tau+3))$, independent of $\tau$.
\end{proof}\\

{\bf Proof of Theorem $\ref{globalu}$:}
From Theorem $\ref{great}$, we have global existence of the solution to $(\ref{sy5})$. Now suppose $1\leq j\leq k$ and $1\leq i\leq m$, such that $V_{i,j}1$ and $V_{i,j}2$ hold. Our solution is uniformly bounded for $t\in [0,1]$. Combining this with Lemma $\ref{global_p}$, if $p_1=\frac{n+3}{n+2}$ then there exists $C_{p_1}\geq 0$  such that \[\Vert u_j\Vert_{p_1,\Omega\times (\tau, \tau+3)}, \Vert u_j\Vert_{p_1,M\times(\tau, \tau+3)}, \Vert v_i\Vert_{p_1,M\times (\tau,\tau+3)}\leq C_{p_1}\]  for all $\tau\geq 0$. Now applying Lemma $\ref{global_p}$, for $q=\frac{n+3}{n+2}$ and $p_2= \left(\frac{n+3}{n+2}\right)^2$, and the uniform bound for our solution for $t\in[0,1]$, there exists $C_{p_2}\geq 0$ such that  \[\Vert u_j\Vert_{p_2,\Omega\times (\tau, \tau+3)}, \Vert u_j\Vert_{p_2,M\times(\tau, \tau+3)}, \Vert v_i\Vert_{p_2,M\times (\tau,\tau+3)}\leq C_{p_2}\] for all $\tau\geq 0$. Repeating the process above, if $p>1$, and $\tau\geq 0$ then there exists $C_{p}\geq 0$, independent of $\tau$, such that \begin{align}\label{p*}\Vert u_j\Vert_{p,\Omega\times (\tau, \tau+3)}, \Vert u_j\Vert_{p,M\times(\tau, \tau+3)}, \Vert v_i\Vert_{p,M\times (\tau,\tau+3)}\leq C_{p} \end{align}

From the hypotheses we are assured that we have estimate $(\ref{p*})$ for each component of $u$ and $v$. Now we convert these $L_p$ estimates to sup norm estimates. For that purpose, let $\tau\geq 0$ and define a cut off function $\psi\in{C_{0}}^{\infty}(\mathbb{R},[0,1])$  such that $\psi(t)=1$ for all $t\geq \tau+1$ and $\psi(t)=0$ for all $t\leq \tau$. In addition, define $
\hat v_i(x,t)= \psi(t) v_i(x,t) $ and
$ \hat u_j(x,t)= \psi(t)  u_j(x,t)$.
 From construction, $ \hat v_i(x,t)=  v_i(x,t)$ and $\hat u_j(x,t)=  u_j(x,t)$
 for all $(x,t)\in M\times (\tau+1, \tau+3)$ and $(x,t)\in \Omega\times (\tau+1, \tau+3)$ respectively. Also $\hat u_j$ and $\hat v_i$ satisfy the system 
\begin{align} {\hat {u_j}}_t&=d_j\Delta \hat {u_j}+\psi'(t) u_j(x,t)+\psi(t) H_j(u)
&( x,t)\in \Omega\times (\tau,\tau+3)\nonumber
\\ \hat {v_i}_t&=\tilde d_i\Delta_M \hat {v_i}+\psi'(t) v_i(x,t)+ \psi(t) F_i(u,v) & (x,t)\in M\times (\tau,\tau+3)\nonumber\\ d_j\frac{\partial \hat {u_j}}{\partial \eta}&=\psi(t) G_j(u,v) &(x,t)\in M\times (\tau,\tau+3)\nonumber \\
\hat {u_j}&=0 &x\in\Omega ,\quad t=\tau\nonumber\\ \hat {v_i}&=0 &x\in M ,\quad t=\tau\nonumber\end{align} 
From $(V_{i,j}2)$ and $(V_{i,j}3)$, $F$ and $G$ are polynomially bounded above. So, consider the system
\begin{align}\label{k} {\hat {u_j}}_t&=d_j\Delta \hat {u_j}+\psi'(t) u_j(x,t)+\psi\beta (u_j+1)
&( x,t)\in \Omega\times (\tau,\tau+3)\nonumber
\\ \hat {v_i}_t&=\tilde d_i\Delta_M \hat {v_i}+\psi'(t) v_i(x,t)+\psi K_f(|u|+|v|+1) & (x,t)\in M\times (\tau,\tau+3)\nonumber\\ d_j\frac{\partial \hat {u_j}}{\partial \eta}&=\psi K_g(u_j+v_i+1) &(x,t)\in M\times (\tau,\tau+3) \\
\hat {u_j}&=0 &x\in\Omega ,\quad t=\tau\nonumber\\ \hat {v_i}&=0 &x\in M ,\quad t=\tau\nonumber\end{align} 
Note that $u_j\leq \hat u$ and $v_i\leq \hat v$ for all $t> 0$. From $(\ref{p*})$, if $q>n+1$, then $K_f(|u|+|v|+1)^l$ and $K_g(u+v+1)$ have $L_{q}(M\times(\tau,\tau+3))$ bounds independent of $\tau\geq 0$.  Using Lemma $\ref{bw}$, the solution of system $(\ref{k})$ is sup norm bounded. Therefore, by the comparision principle \cite{RefWorks:62}, the solution of $(\ref{sy5})$ is uniformly bounded. 
$\square$\\

{\bf Proof of Corollary $\ref{adventure23}$:} For simplicity, take $\sigma=1$ in $(V_{i,j}1)$. Let $\tau\geq 0$, and consider the system \begin{align}\label{aj5} \varphi_t&=- d_j\Delta \varphi
&( x,t)\in \Omega\times (\tau, \tau+1)\nonumber
\\ d_j\frac{\partial \varphi}{\partial \eta}&= 1 &(x,t)\in M\times (\tau, \tau+1)\\
\varphi&= \varphi_{\tau+1} &x\in\Omega ,\quad t=\tau+1\nonumber
\end{align}
where $d_j>0$, and $\varphi_{\tau+1}\in C^{2+\varUpsilon}(\overline\Omega)$ for some $\varUpsilon>0$, is nonnegative and satisfies the compatibility condition \[ d_j\frac{\partial\varphi_{\tau+1}}{\partial\eta}=1 \quad \text{on} \ M\times \lbrace \tau+1\rbrace\] From Theorem 5.3 in chapter 4 of \cite{RefWorks:65}, $\varphi\in C^{2+\varUpsilon,1+\frac{\varUpsilon}{2}}(\overline\Omega\times[\tau, \tau+1])$ and therefore $\varphi\in C^{2+\varUpsilon,1+\frac{\varUpsilon}{2}}(M\times[\tau, \tau+1])$. Also $\frac{\partial\varphi}{\partial\eta}\geq 0$. Therefore the maximum principle implies $\varphi\geq 0$. Now having enough regularity for $\varphi$ on $M\times[\tau,\tau+1]$, consider \[\Delta_M \varphi=-\frac{1}{\sqrt {det\  g}}\partial_j (g^{ij}\sqrt{det\ g}\ \partial_i\varphi)\] where $g$ is the metric on $M$ and $g^{i,j}$ is $i$th row and $j$th column entry of the inverse of the matrix associated with the metric $g$. Further, let $\tilde\vartheta=-\varphi_t-\tilde d_i\Delta_M\varphi$. 
Then,
\begin{align*}
&\int_\tau^{\tau+1}\int_{M} v_i\tilde\vartheta  =\int_\tau^{\tau+1}\int_{\Omega} u_j(-\varphi_t-d_j\Delta\varphi) +\int_\tau^{\tau+1}\int_{M} v_i(-\varphi_t-\tilde d_j\Delta_M \varphi) \\
& =\int_\tau^{\tau+1}\int_{\Omega} \varphi ({u_j}_t-d_j\Delta u_j) +\int_\tau^{\tau+1}\int_{M} \varphi({v_i}_t-\tilde d_i\Delta_M v_i) -d_j\int_\tau^{\tau+1}\int_{M} u_j\frac{\partial\varphi}{\partial\eta}+d_j\int_\tau^{\tau+1}\int_{M} \frac{\partial u_j}{\partial\eta}\varphi\\ &\quad+\int_{\Omega} u_j(x,\tau)\varphi(x,\tau)+\int_{M} v_i(\zeta,\tau)\varphi(x,\tau)-\int_{\Omega} u_j(x,\tau+1)\varphi(\cdot, \tau+1)-\int_{M} v_i(\zeta,\tau+1)\varphi (\cdot, \tau+1)
\end{align*}
Using $d_j\frac{\partial\varphi}{\partial\eta}=1$ and $(V_{i,j}1)$  
\begin{align}\label{bu}
\int_\tau^{\tau+1}\int_{M} u_j&\leq \int_{\Omega} u_j(x,\tau)\varphi(x,\tau)+\int_{M} v_i(\zeta,\tau)\varphi(x,\tau)-\int_\tau^{\tau+1}\int_{M} v_i\tilde\vartheta  
\end{align}
Now, integrating the $u_j$ equation over $\Omega$ and the $v_i$ equation over $M$, 
\begin{align*}
\frac{d}{dt}\left(\int_{\Omega} u_j+ \int_{M} v_i\right)&=d_j\int_{\Omega} \Delta u_j +\int_{\Omega}H_j(u) +\tilde d_i\int_{M} \Delta v_i + \int_{M} F_j(u,v)\nonumber \\
&\leq \int_{M} (F_j(u,v)+G_i(u,v)) \nonumber\\
&\leq 0
\end{align*}
From Gronwall's inequality, for all $t\geq 0$\begin{align}\label{intau}
\int_{\Omega} u_j (x,t)+ \int_{M} v_i(\zeta, t)  &\leq  \int_{\Omega}u_j(x,0)+\int_{M}v_i(\zeta, 0)
\end{align}
So, from $(\ref{intau})$ and $(\ref{bu})$, we get
\begin{align*}
 \int_{\tau}^{\tau+1}\int_{M} u_j(\zeta,\tau)&\leq \Vert u_j(\cdot,\tau)\Vert_{1,\Omega}\cdot\Vert\varphi(\cdot,\tau)\Vert_{\infty,\Omega}+\Vert v_i(\cdot,\tau)\Vert_{1,M}\Vert\tilde\vartheta\Vert_{\infty,M\times(\tau, \tau+1)}\nonumber\\&\nonumber\quad+\Vert v_i(\cdot,\tau)\Vert_{1,M}\cdot \Vert\varphi(\cdot,\tau)\Vert_{\infty, M}\nonumber\\&\leq K
\end{align*}
for all $\tau\geq 0$. This estimate and  $(\ref{intau})$ give the $L_1$ bounds needed in Theorem $\ref{globalu}$. Therefore, the solution to $(\ref{sy5})$ is uniformly bounded. $\square$

\section{Examples}
In this section we give some examples to support our theory.\\
\begin{example}
We show that following version of the Brusselator has a uniformly bounded solution.
\begin{align}\label{brus}
u_t &= \Delta u, &  x\in\Omega,t>0 \nonumber\\
v_t &= d\Delta_M v+B-(A+1)v+v^2u, & x\in\partial\Omega,t>0 \nonumber\\
\frac{\partial u}{\partial\eta} &=Av-v^2u, & x\in\partial\Omega,t>0 \\
u&=u_0 & x\in\Omega\nonumber\\
v&=v_0 & x\in\partial\Omega \nonumber
\end{align}
Here $ A, B, d>0$, the initial data is smooth and non negative, and $u_0$ and $v_0$ satisfy the compatibility condition. Note that $f+g \nleq 0$, so, we cannot directly apply Corollary $\ref{adventure23}$. Comparing $(\ref{brus})$ with $(\ref{sy5})$, we have\[ F(u,v)= B-(A+1)v+v^2u, \quad G(u,v)=Av-v^2u \quad \text{and}\quad H(u)=0\] It is easy see that $F, G$ and $H$ are quasipositive, $F+G\leq B$ and $G\leq A(u+v+1)$. So, the hypothesis of Theorem $\ref{great}$ are satisfied. As a result, we have global existence of solutions to $(\ref{brus})$. Note that the solution is bounded for $0\leq t<1$. Also, applying comparison principle to the $v$ equation, implies $v(x,t)\geq y(t)$, where $y$ solves \[y'(t)=B-(A+1)y, \quad y(0)=0\]
After simple calculation $y(t)=\frac{B}{A+1}(1-e^{-(A+1)t})$, therefore \[v(x,t)\geq K_1=\frac{B}{A+1}(1-e^{-(A+1)}) \quad\text{for all}\  t\geq 1\]
We get an $L_2(\Omega)$ estimate on $u(\cdot,t)$, independent
of $t$, as follows. Multiply the $u$ equation by $u$, and integrate
over $\Omega$. This gives 
\[
\frac{d}{dt}\int_{\Omega}u^{2}dx+\int_{\Omega}|\nabla u|^{2}dx=\int_{\partial\Omega}(2Au\,v-2u^{2}v^{2})d\sigma
\]
Using the lower bound on $v$, we have 
\[\frac{d}{dt}\int_{\Omega}u^{2}dx+\int_{\Omega}|\nabla u|^{2}dx+K_{1}^{2}\int_{\partial\Omega}u^{2}d\sigma\le\int_{\partial\Omega}(2Au\,v-u^{2}v^{2})d\sigma\]
for all $t\ge1$. From the compact embeddings of $H^{1}(\Omega)$
into each of $L_2(\Omega)$ and $L_2(\partial\Omega)$, there
is a constant $\delta>0$, independent of $u$, such that 
\[\delta\int_{\Omega}u^{2}dx\le\int_{\Omega}|\nabla u|^{2}dx+K_{1}^{2}\int_{\partial\Omega}u^{2}d\sigma\]
Furthermore, since $uv\ge0$, there is a constant $M>0$ so that $2Auv-(uv)^{2}\le\frac{M}{|\partial\Omega|}$.
As a result, we have 
\[\frac{d}{dt}\int_{\Omega}u^{2}dx+\delta\int_{\Omega}u^{2}dx\le M\]
 for all $t\ge1$. Therefore, there exists $K_{2}>0$ (depending on
$M$, $\delta$ and $u_{0}$) such that 
\[\int_{\Omega}u(x,t)^{2}dx\le K_{2}\]
for all $t\ge0$. Also, this certainly implies there exists $K_{3}>0$
so that 
\[\int_{\Omega}u(x,t)dx\le K_{3}\]
for all $t\ge0$. Now to obtain an $L_1(\partial\Omega)$ estimate on $v(\cdot,t)$
independent of $t\ge0$. To this end, integrate the $u$ equation
over $\Omega$, and the $v$ equation over $\partial\Omega$, and
sum the equations to obtain 
\[
\frac{d}{dt}\left(\int_{\Omega}udx+\int_{\partial\Omega}vd\sigma\right)\le\int_{\partial\Omega}(B-v)d\sigma
\]
As a result, 
\[
\frac{d}{dt}\left(\int_{\Omega}udx+\int_{\partial\Omega}vd\sigma\right)\le K_{3}+B|\partial\Omega|-\left(\int_{\Omega}udx+\int_{\partial\Omega}vd\sigma\right)
\]
So, once again, a comparison principle can be used to find a bound
for $\int_{\Omega}udx+\int_{\partial\Omega}vd\sigma$, independent
of $t\ge0$. Consequently, from the $L_1(\Omega)$ bound on $u$
above, there is a constant $K_{4}>0$ so that 
\[
\int_{\partial\Omega}v(x,t)d\sigma\le K_{4}
\]
for all $t\ge0$. So, at this point, if we set $K_{5}=\max\{K_{3},K_{4}\}$,
we know 
\[
\|u(\cdot,t)\|_{1,\Omega},\|v(\cdot,t)\|_{1,\partial\Omega}\le K_{5}
\]
for all $t\ge0$. Finally, to get an estimate in $L_1(\partial\Omega\times(\tau,\tau+1))$
for $u$, independent of $\tau\ge0$. To this end, let $\tau\ge0$. Integrating
the equation for $u$ over $\Omega\times(\tau,\tau+1)$ gives 
\[
\int_{\Omega}u(x,\tau+1)dx-\int_{\Omega}u(x,\tau)dx\le A\int_{\tau}^{\tau+1}\int_{\partial\Omega}v(x,t)d\sigma dt-K_{1}^{2}\int_{\tau}^{\tau+1}\int_{\partial\Omega}u(x,t)d\sigma dt
\]
implying, 
\[
\int_{\tau}^{\tau+1}\int_{\partial\Omega}u(x,t)d\sigma dt\le\frac{(A+1)K_{5}}{K_{1}^2}
\]
Therefore, Theorem $\ref{globalu}$ guarantees a uniform sup norm bound for $u$
and $v$.
\end{example}
\vspace{.1cm}
\begin{example2}
Consider the model considered by R\"{a}tz and R\"{o}ger \cite{RefWorks:99} for signaling networks. They formulated a mathematical model that couples reaction-diffusion in the inner volume to a reaction-diffusion system on the membrane via a flux condition. More specifically, consider the system $(\ref{sy5})$ with $k=1$ and $m=2$, where \[{G(u,v)}=-q=-b_6 \frac{|B|}{|M|}u(c_{max}-v_1-v_2)_{+}+b_{-6} v_2,\quad  {H(u)}=0\]
 \[ {F(u,v)}=\begin{pmatrix}F_1(u,v)\\F_2(u,v)\end{pmatrix}=\begin{pmatrix}k_1v_2g_0\left(1-\frac{K_5v_1g_0}{1+K_5v_1}\right)+k_2v_2\frac{K_5v_1g_0}{1+K_5v_1}-k_3\frac{v_1}{v_1+k_4}\\-k_1v_2g_0\left(1-\frac{K_5v_1g_0}{1+K_5v_1}\right)-k_2v_2\frac{K_5v_1g_0}{1+K_5v_1}+k_3\frac{v_1}{v_1+k_4} +q \end{pmatrix} \]  Also,
$ u_0=( {u_0}_j)\in W_p^{(2)}(\Omega)$, $v_0= ({v_0}_i)\in W_p^{(2)}(M)$ with $p>n$ are componentwise nonnegative, and $u_0 $  and $v_0$ satisfy the compatibility condition\[ D{\frac{ \partial {u_0}}{\partial \eta}} =G(u_0,v_0)\quad \text{on $M$}\] Here $k_\alpha, K_\alpha, g_0, c_{max}, b_{-6}$ are the same positive constants as described in \cite{RefWorks:99}. We note $F, G$ and $H$ are quasi positive functions. From Theorem $\ref{great}$, this system has a unique componentwise nonnegative global solution. In order to get uniform bounds, note that $G+F_1+F_2=0$ and $H=0$, therefore from an argument similar to Corollary $\ref{adventure23}$, we have uniform $L_1$ estimates for $u,v_1,v_2$. Consequently, Theorem $\ref{globalu}$ implies $(u,v)$ are uniformly bounded.
\end{example2}
\vspace{.1cm}
\begin{example3}
As described in \cite{RefWorks:142}, during bacterial cytokinesis, a proteinaceous contractile, called the $Z$ ring assembles in the cell middle. Positiong the $Z$ ring in the middle of the cell involves two independent processes, referred to as Min system inhibition and nucleoid occlusion \cite{RefWorks:141}. The Min subsystem consists of ATP-bound cytosolic MinD, ADP-bound cytosolic MinD, membrane-bound MinD, cytosolic MinE, and membrane bound MinD:MinE complex. Those are denoted $D_{cyt}^{ATP}$, $D_{cyt}^{ADP}$, $D_{mem}^{ATP}$, $E_{cyt}$, and $E:D_{mem}^{ATP}$, respectively. This essentially constitutes the one dimensional version of the problem. These Min proteins  react with certain reaction rates that are illustrated in Table 1. 
\begin{table}[ht]\label{table:nonlin}
\caption{Reactions and Reaction Rates} 
\centering 
\begin{tabular}{|ccc| }
\hline\hline                       
Chemicals & Reactions & Reaction Rates\\ [0.5ex] 
\hline 
& & \\            
Min D  & $D^{ADP}_{cyt}\xrightarrow{k_{1}} D_{cyt}^{ATP}$ & $ R_{exc}=k_{1}[D_{cyt}^{ADP}]$ \\[1ex]
Min D & $D_{cyt}^{ATP}\xrightarrow{k_{2}} D_{mem}^{ATP}$ & $R_{Dcyt}=k_{2}[D_{cyt}^{ATP}]$\\ [1ex]
&$D_{cyt}^{ATP}\xrightarrow{k_{3}[D_{mem}^{ATP}]} D_{mem}^{ATP}$ & $R_{Dmem}=k_{3}[D_{mem}^{ATP}][D_{cyt}^{ATP}]$\\ [1ex]  
Min E  &$E_{cyt}+D_{mem}^{ATP}\xrightarrow{k_{4}} E:D_{mem}^{ATP}$ &$ R_{Ecyt}=k_{4}[E_{cyt}] [D_{mem}^{ATP}]$\\ [1ex] 
& $E_{cyt}+D_{mem}^{ATP}\xrightarrow{k_{5}[E:D_{mem}^{ATP}]^2} E:D_{mem}^{ATP}$ & $R_{Emem}=k_{5}[D_{mem}^{ATP}][E_{cyt}][E:D_{mem}^{ATP}]^2$\\[1ex]
Min E & $ E:D_{mem}^{ATP} \xrightarrow{k_{6}} E+D_{cyt}^{ADP}$& $R_{exp}=k_{6}[E:D_{mem}^{ATP}]$\\[1ex]
\hline 
\end{tabular}
\end{table} These reactions lead to five component model with $(u,v)=(u_1,u_2,u_3,v_1,v_2)$, where
\[u=\begin{pmatrix}u_1\\u_2\\u_3\end{pmatrix}=\begin{pmatrix} \left[D_{cyt}^{ATP}\right]\\ \left[D_{cyt}^{ADP}\right]\\ \left[E_{cyt}\right] \end{pmatrix}, {v}=\begin{pmatrix}v_1\\v_2\end{pmatrix}=\begin{pmatrix}\left[D_{mem}^{ATP}\right]\\ \left[E:D_{mem}^{ATP}\right]\end{pmatrix} \]
\[\tilde D=\begin{pmatrix} \sigma_{Dmem} & 0 \\ 0 &  \sigma_{E:Dmem} \end{pmatrix},\quad  D=\begin{pmatrix} \sigma_{Dcyt} & 0 & 0\\ 0 &  \sigma_{ADyct} & 0\\0 & 0 & \sigma_{Ecyt} \end{pmatrix} \]\\
 \[{G(u,v)}=\begin{pmatrix}G_1(u,v)\\G_2(u,v)\\G_3(u,v)\end{pmatrix}=\begin{pmatrix}- R_{Dcyt}-R_{Dmem}\\ R_{exp}\\ R_{exp}-R_{Ecyt}-R_{Emem}\end{pmatrix}=\begin{pmatrix}- k_2 u_1-k_3v_1u_1\\ k_6v_2\\ k_6v_2-k_4u_3v_1-k_5v_1u_3{v_2}^2\end{pmatrix}, \]
 \[ {F(u,v)}=\begin{pmatrix}F_1(u,v)\\F_2(u,v)\end{pmatrix}=\begin{pmatrix} R_{Dcyt}+R_{Dmem}-R_{Ecyt}-R_{Emem}\\ -R_{exp}+R_{Ecyt}+R_{Emem}\end{pmatrix}= \begin{pmatrix}k_2u_1+k_3v_1u_1-k_4u_3v_1-k_5v_1u_3{v_2}^2\\ -k_6v_2+k_4u_3v_1+k_5v_1u_3{v_2}^2\end{pmatrix}, \] \[ {H(u)}=\begin{pmatrix}H_1(u)\\H_2(u)\\H_3(u)\end{pmatrix}=\begin{pmatrix} R_{exc}\\ -R_{exc}\\ 0\end{pmatrix}=\begin{pmatrix} k_1u_2\\ -k_1u_2\\ 0\end{pmatrix}, \] and
$ u_0=( {u_0}_j)\in W_p^{2}(\Omega)$, $v_0= ({v_0}_i)\in W_p^{2}(M)$ are componentwise nonnegative functions with $p>n$.  Also, $u_0 $  and $v_0$ satisfy the compatibility condition\[ D{\frac{ \partial {u_0}}{\partial \eta}} =G(u_0,v_0)\quad \text{on $M$}\] Here expressions of the form $k_{\alpha}$ and $\sigma_{\beta}$ are positive constants. Note $F, G$ and $H$ are quasi positive functions. In the multidimensional setting, the concentration densities satisfy the reaction-diffusion system given by $(\ref{sy5})$.
From Theorem $\ref{great}$, this system has a unique componentwise nonnegative global solution. In this example, if we take two specific components at a time, we are able to obtain uniform bounds for each of the components. For that purpose we apply our results to $(u_3,v_2)$, $u_2$ and $(u_1,u_2,v_1)$. 

Consider $(u_3,v_2)$, it is easy to see that for $j=3$ and $i=2$, the hypothesis of Corollary $\ref {adventure23}$ is satisfied, since $G_3+F_2\leq 0$ and $H_3=0$. As a result, $u_3$, $v_2$ are uniformly bounded. Integrating equation $u_2$ over $\Omega\times(\tau,\tau+1)$ and using uniform estimates of $v_2$, we get uniform $L_1$ estimates for $u_2$ and from Theorem $\ref{globalu}$, $u_2$ is uniformly bounded on $\Omega\times(0,{\infty})$. Finally, consider $(u_1, u_2,v_1)$. Integrating $u_1, u_2$ equation over $\Omega$ and adding the sum of them with $v_1$ equation over $\partial\Omega$, and further using uniform estimates of $v_2$, we get uniform $L_1$ estimates for $u_1, v_1$. Therefore from Theorem $\ref{globalu}$, $u_1$ and $v_1$ are also uniformly bounded in time.\\

\end{example3}
\vspace{.1cm}

\end{document}